\documentclass{article}
\usepackage{amsmath}
\usepackage{amsfonts}
\usepackage{amssymb}
\usepackage{graphicx}
\usepackage{mathrsfs}
\usepackage{latexsym}
\usepackage{xy} \xyoption{all}
\usepackage{amsthm}
\usepackage{epsfig}
\RequirePackage{color}
\usepackage{multicol}

\newtheorem{theorem}{Theorem}[subsection]

\newtheorem{corollary}[theorem]{Corollary}

\newtheorem{definition}[theorem]{Definition}
\newtheorem{example}[theorem]{Example}

\newtheorem{lemma}[theorem]{Lemma}

\newtheorem{proposition}[theorem]{Proposition}
\newtheorem{remark}[theorem]{Remark}

\begin{document}

\title{A Survey on The Ideal Structure of Leavitt Path Algebras}
\author{M\"{u}ge KANUN\.{I} $^{(1)}$, Suat SERT $^{(1)}$\\$^{(1)\text{ }}$Department of Mathematics, D\"{u}zce University, \\Konuralp 81620 D\"{u}zce, Turkey.\\E-mail: mugekanuni@duzce.edu.tr\\E-mail: suatsert@gmail.com\\}
\date{}
\maketitle

\begin{abstract}
There is an extensive recent literature on the graded, non-graded, prime, primitive, maximal ideals of Leavitt path algebras.
In this introductory level survey, we will be giving an overview of different types of ideals and the correspondence between the lattice of ideals and the lattice of hereditary and saturated subsets of the graph over which the Leavitt path algebra is constructed.  
\end{abstract}
\section*{Introduction}
Leavitt path algebras are introduced independently by Abrams and Aranda Pino in \cite{AA} and by Ara, Moreno and Pardo in \cite{AMP} around 2005. When the Leavitt path algebra is defined over the complex field it is the dense subalgebra of the graph C$^{\ast}$-algebra. (For a comprehensive survey on the graph C$^{\ast}$-algebras by Raeburn, see \cite{Raeburn}). This close connection between algebra and analysis, flourished with many similar results on the algebraic and analytic structures. A survey article
by Abrams \cite{A-Decade} summarized this interaction, also listed the similarities/differences of algebraic and analytic results giving an extensive list of references. This topic attracted the interest of many mathematicians immediately as the structure reveals itself in the graph properties on which it is constructed. Leavitt path algebras produced examples to answer some well-known open problems. Hence, hundreds of papers are published within a decade. 

For a detailed discussion on Leavitt path algebras, interactions with various topics, we refer the interested reader to a well-written introductory level book published in 2017 by Abrams, Ara and Siles Molina \cite{AAS} which covers most of the literature. 

Our main aim in this article is to focus only on the prime, primitive and maximal two-sided ideals of Leavitt path algebras over a field, we gather and cite the known results that are either included in the book \cite{AAS} or some recent {\it to appear} results \cite{EK}, \cite{Ranga3}. To keep the survey short and to avoid the overlap with other expository papers, we did not include many other important and interesting topics in the ideal structure of Leavitt path algebras. We also did not extend the discussion to the results on Leavitt path algebras over commutative rings.   

\section{Preliminaries} 
The first section consists of preliminary definitions all of which can be found in the book \cite{AAS}.
\subsection{Graph Theory}
We first start with the basic definitions on graphs that is the main discrete structure of our interest. In this paper, $E = (E^0, E^1, s, r)$ will denote a directed graph with vertex set $E^0$, edge set $E^1$, source function $s$, and range function $r$. In particular, the source vertex of an edge $e$ is denoted by $s(e)$, and the range vertex by $r(e)$. The graph $E$ is called {\it finite} if both $E^0$ and $E^1$ are finite sets, and called {\it row-finite} if every vertex emits only finitely many edges. A vertex which emits infinitely many edges is called 
an {\it infinite emitter}. A {\it sink} is a vertex $v$ for which the set $s^{-1}(v) = \{e\in E^1 \mid s(e) = v\}$ is empty, i.e. emits no edges. A vertex is a {\it regular} vertex if it is neither a sink nor an infinite emitter. 

A proper path $\mu$ is a sequence of edges $\mu=e_1e_2...e_n$ such that $s(e_i)=r(e_{i-1})$ for $i=2,...,n$. Any vertex is considered to be a trivial path of length zero. 
The \textit{length of a path $\mu$} is the number of edges forming the path, i.e. $l(\mu)=n$ and the 
set of all paths is denoted by Path$(E)$. 
If $n=l(\mu)\geq 1$, and $v=s(\mu)=r(\mu)$, then $\mu$ is called a \textit{closed path based at v}. Again, $\mu$ is a \textit{closed simple path based at v} if $s(e_j)\neq v$ for every $j>1$. 
If $\mu=e_1e_2...e_n$ is a closed path based at $v$ and $s(e_i)\neq s(e_j)$ for every $i \neq j$, then $\mu$ is called a {\it cycle based at $v$}. An \textit{exit }for a path $\mu=e_{1}\dots e_{n}$ is an edge $e$ such that
$s(e)=s(e_{i})$ for some $i$ and $e\neq e_{i}$.
	A cycle of length $1$ is called a {\it loop}. A graph $E$ is said to be {\it acyclic} in case it does not have any closed paths 
	based at any vertex of $E$. 

There are some graph properties that deserves to be named which will be used in the sequel.  

\begin{definition}\label{M(u)} \rm
	For $v,w \in E^0$, we write 
	$v \geq w$ in case there is a path $\mu \in Path(E)$ such that $s(\mu)=v$ and $r(\mu)=w$.\\
	
	If $v \in E^0$ then the {\it tree of $v$}, denoted $T(v)$, is the set
	$$T(v)=\{ w \in E^0 \mid v \geq w\}.$$ 
	Also, define $M(v)=\{w\in E^{0}: w\geq v\}.$
\end{definition}

\begin{definition}\rm
	A graph $E$ satisfies {\it Condition $(K)$} if for each $v \in E^0$ which lies on a closed simple path, there exist at least two distinct closed simple paths $\alpha,\beta$ based at $v$. \\
	A graph $E$ satisfies {\it Condition $(L)$} if every cycle in $E$ has an exit.\\
	%2013 Ranga Theory of Prime Ideals  LPA %%%%%
	A cycle $c$ in a graph $E$ is called a {\it cycle without $K$}, if no vertex on $c$ is the base of another distinct cycle in $E$ (where distinct cycles possess different sets of edges).\\ %and different sets of vertices).\\
	A graph $E$ satisfies the \textit{Countable Separation Property}, if there exists a countable set $S$ of vertices in $E$ such that, for each vertex $u \in E$, there exists $w \in S$ for which $u \geq w$. \\
	A graph $E$ is said to be \textit{countably directed} if there is a non-empty at most countable subset $S$ of $E^0$ such that, 
for any two $u, v \in E^0$, there is a $w \in S$ such that $u \geq w$ and $v \geq w$.
\end{definition}

\begin{definition}\rm
	Let $E$ be a graph, and $H \subseteq E^0$.
	 $H$ is {\it hereditary} if whenever $v \in H$ and $w \in E^0$ for which $v \geq w$, then $w \in H$. \\
	$H$ is {\it saturated} if whenever a regular vertex $v$ has the property that $\{r(e)| e \in E^1, s(e)=v\} \subseteq H$, then $v \in H$. 
	
We denote $\mathscr{H}_E$ the set of those subsets of $E^0$ which are both hereditary and saturated.

\end{definition}

For a given graph, there are many different new graph constructions that play a role in the ideal theory of Leavitt path algebras.
\begin{definition}\rm \label{graph}
	\textbf{(The restriction graph $E_H$)} Let $E$ be an arbitrary graph, and let $H$ be a hereditary subset of $E^0$. We We denote bydenote by $E_H$ the restriction graph:
	$$E_{H}^{0}:=H,\quad E_{H}^{1}:=\{e \in E^1 | s(e) \in H \},$$
	and the source and range functions in $E_H$ are the source and range functions in $E$, restricted to $H$.

	\textbf{(The quotient graph by a hereditary subset $E/H$)} Let $E$ be an arbitrary graph, and let $H$ be a hereditary subset of $E^0$. We denote by $ E / H $ the \textit{quotient graph of $E$ by $H$}, defined follows:
	$$(E/H)^0 = E^0 \backslash H , \text{ and } (E/H)^1 = \{e \in E^1 | r(e) \notin H \}.$$
	The range  and source functions for $E/H$ are defined by restricting the range and source functions of $E$ to $(E/H)^1$.
	
	\textbf{(The hedgehog graph for a hereditary subset $F_E(H)$)} Let $E$ be an arbitrary graph. Let $H$ be a nonempty hereditary subset of $E^0$. We denote by $F_E(H)$ the set
	$$F_E(H)=\{ \alpha \in Path(E) | \alpha=e_1....e_n , \text{ with } s(e_1) \in E^0 \backslash H, r(e_i) \in E^0 \backslash H \text{ for all }$$ $$1 \leq i < n, \text{ and } r(e_n) \in H \}$$
	We denote by $\overline{F}_E(H)$ another copy of $F_E(H)$. If $\alpha \in F_E(H)$, we will write $\overline{\alpha}$ to refer to a copy of $\alpha$ in $\overline{F}_E(H)$. We define the graph $_HE=(_HE^0,_HE^1,s',r')$ as follows:
	$$_HE^0=H \cup F_E(H), \quad \text{ and } \quad _HE^1=\{e \in E^1 | s(e) \in H \} \cup \overline{F}_E(H).$$
	The source and range functions $s'$ and $r'$ are defined by setting $s'(e)=s(e)$ and $r'(e)=r(e)$ for every $e \in E^1$ such that $s(e) \in H$; and by setting $s'(\overline{\alpha})=\alpha$ and $r'(\overline{\alpha})=r(\alpha)$ for all $\overline{\alpha} \in \overline{F}_E(H)$.
	
	Intuitively, $F_E(H)$ can be viewed as $H$, together with a new vertex corresponding to each path in $E$ which ends at a vertex in $H$, but for which none of the previous edges in the path ends at a vertex in $H$. For every such new vertex, a new edge is added going into $H$. 
In $F_E(H)$, the only paths entering the subgraph $H$ have common length 1; (the new graph looks like a hedgehog where the body is $H$ and the quills are the edges into $H$). 
\end{definition}

\begin{example}\label{Ekornek} 
	\rm
	Consider the graph $E$ below and take the hereditary saturated subset $H= \{v,w \}$,
	\medskip
	\qquad \qquad \qquad 
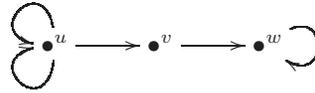
	$$ \xymatrix{
		{\bullet}^u  \ar@(u,l) \ar@(d,l)   \ar@{->}[r]  & {\bullet}^v \ar@{->}[r]
		& {\bullet}^w \ar@(ur,dr) }
	$$ 
	\medskip
	The restriction graph is 
	
	\medskip
	
	\qquad \qquad \qquad $E_H$
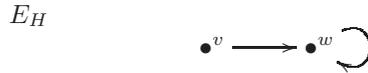
	$$ \xymatrix{
		{\bullet}^v \ar@{->}[r]
		& {\bullet}^w \ar@(ur,dr) }
	$$ 
	\medskip
	
	The quotient graph $E/H$ is
	
	\medskip
	\qquad \qquad \qquad $E/H$
	$$ \xymatrix{
		{\bullet}^u  \ar@(u,l) \ar@(d,l)    }
	$$ 
	%\medskip
	
	%The hedgehog graph ${}_HE$ is 

\end{example}

\medskip 

\begin{example}\label{Ekornek2} 
\rm
Consider the graph $E$ below and take the hereditary saturated subset $H= \{v,w \}$,
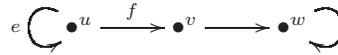
$$ \xymatrix{
		{\bullet}^u  \ar@(ul,dl)_e  \ar@{->}[r]^f  & {\bullet}^v \ar@{->}[r]
		& {\bullet}^w \ar@(ur,dr) }
	$$ 
The hedgehog graph ${}_HE$ is 
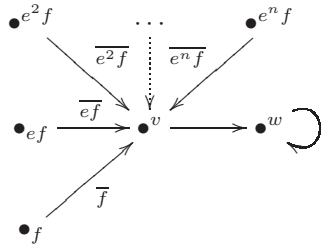
$$ \xymatrix{
		{\bullet}^{e^2f}  \ar@{->}[dr]^{\overline{e^2f}} 
		& %{\bullet}^{x_n} 
		\cdots \ar@{.>}[d] 
		& {\bullet}^{e^nf} 
		\ar@{->}[dl]_{\overline{e^nf}} \\
		{\bullet}_{ef}  \ar@{->}[r]^{\overline{ef}}  & {\bullet}^v \ar@{->}[r]
		& {\bullet}^w \ar@(ur,dr) \\ 
		{\bullet}_{f}  \ar@{->}[ur]_{\overline{f}} & & }
	$$ 
\end{example}

\medskip

When we have infinite emitters in a graph, the graph is not row-finite and we need to introduce the notion of breaking vertices.
\begin{definition}\rm
	Let $E$ be an arbitrary graph and $K$ be any field. Let $H$ be a hereditary subset of $E^0$, and let $v \in E^0$. We say that $v$ is a \textit{breaking vertex of $H$} if $v$ belongs to the set
	\[
	B_H:=\{ v \in E^0 \backslash H | \text{ v is an infinite emitter and } 0 < | s^{-1} (v) \cap r^{-1} (E^0 \backslash H)| < \infty \}.
	\]
	In words, $B_H$ consists of those vertices of $E$ which are infinite emitters, which do not belong to $H$, and for which the ranges of the edges they emit are all, except for a finite (but nonzero) number, inside $H$. For $v \in B_H$, we define the element $v^H$ of $L_K(E)$ by setting
	\[
	v^H:=v - \displaystyle\sum_{e\in s^{-1}(v) \cap r^{-1}(E^0 \backslash H)} ee^{\ast} .
	\]
	We note that any such $v^H$ is homogeneous of degree $0$ in the standard $\mathbb{Z}$-grading on $L_K(E)$. For any subset $S \subseteq B_H$, we define $S^H \subseteq L_K(E)$ by setting $S^H = \{v^H | v \in S\}$. Given a hereditary saturated subset $H$ and a subset $S \subset B_H$, $(H, S)$ is called an \textit{admissible pair}. Given an admissible pair $(H, S)$, the ideal generated by $H \cup S^H$ is denoted by $I(H, S)$.
\end{definition}
Now, the new graph constructions that we defined in Definition \ref{graph}, can be extended to graphs with infinite emitters. 
\begin{definition}\rm
	\textbf{(The quotient graph $E/(H,S)$)} Let $E$ be an arbitrary graph, $H \in \mathscr{H}_E$, and $S \subseteq B_H$. We denote by $E/(H,S)$ the \textit{quotient graph of $E$ by $(H,S)$}, defined as follows:
	$$(E/(H,S))^0 = (E^0 \backslash H) \cup \{ v' | v \in B_H \backslash S \},$$	
	$$(E/(H,S))^1 = \{e \in E^1 | r(e) \notin H\} \cup \{ e' | e \in E^1 \text{ and } r(e) \in B_H \backslash S \},$$
	and range and source maps in $E/(H,S)$ are defined by extending the range and source maps in $E$ when appropriate, and in addition setting $s(e')=s(e)$ and $r(e')=r(e)'$.

	\textbf{(The generalized hedgehog graph construction $_{(H,S)}E$)} Let $E$ be an arbitrary graph, $H$ a nonempty hereditary subset of $E$, and $S \subseteq B_H$. We define
	$$F_1(H,S):=\{\alpha \in Path(E) | \alpha=e_1...e_n, r(e_n) \in H \text{ and } s(e_n) \notin H\cup S\}, \text{ and }$$ 
	$$F_2(H,S):=\{\alpha \in Path(E)|\quad |\alpha|\geq 1 \text{ and } r(\alpha) \in S\}.$$
	For $i=1,2$ we denote a copy of $F_i(H,S)$ by $\overline{F}_i(H,S)$. We define the graph $_{(H,S)}E$ as follows:
	$$_{(H,S)}E^0:=H \cup S\cup F_1(H,S)\cup F_2(H,S), \text{ and }$$
	$$_{(H,S)}E^1:=\{e \in E^1 | s(e) \in H\} \cup \{e \in E^1 | s(e) \in S \text{ and } r(e) \in H\} \cup \overline{F}_1(H,S) \cup \overline{F}_2(H,S).$$\\
	The range and source map for $_{(H,S)}E$ are described by extending $r$ and $s$ to $_{(H,S)}E^1$, and by defining $r(\overline{\alpha})=\alpha$ and $s(\overline{\alpha})=\alpha$ for all $\overline{\alpha}\in \overline{F}_1(H,S)\cup \overline{F}_2(H,S)$.
\end{definition}
\begin{definition}\rm 
  A graph $F$ is a \textit{subgraph} of a graph $E$, if $F^0 \subset E^0$ and $F^1 \subset E^1$ where 
  for any $f \in F^1$, $s(f), r(f) \in F^0$.
 
	A subgraph $F$ of a graph $E$ is called \textit{full} in case for each $v,w \in F^0$,
	$$\{f \in F^1 | s(f)=v, r(f)=w\}=\{e \in E^1 | s(e)=v, r(e)=w\}.$$
	In other words, the subgraph $F$ is full in case whenever two vertices of $E$ are in the subgraph, then all of the edges connecting those two vertices in $E$ are also in $F$.
	 
	A non-empty full subgraph $M$ of $E$ is a \textit{maximal tail} if it satisfies the following properties:
	\begin{enumerate}
		\item [($MT-1$)] If $v\in E^0 , w \in M^0$ and $v \geq w$, then $v \in M^0$;
		\item [($MT-2$)] If $v \in M^0$ and $s_{E}^{-1}(v) \neq \emptyset$, then there exists $e \in E^1$ such that $s(e)=v$ and $r(e) \in M^0$; and
		\item [($MT-3$)] If $v,w \in M^0$, then there exists $y \in M^0$ such that $v \geq y$ and $ w \geq y$.
	\end{enumerate}
Condition $MT-3$ is now more commonly called \textit{downward directedness} in literature, however we will use the term $MT-3$ for consistency throughout the text.
\end{definition}

\subsection{Leavitt Path Algebra}
\begin{definition}\rm
Given an arbitrary graph $E$ and a field $K$, the {\it Leavitt path algebra
}$L_{K}(E)$ is defined to be the $K$-algebra generated by a set $\{v:v\in
E^{0}\}$ of pair-wise orthogonal idempotents together with a set of variables
$\{e,e^{\ast}:e\in E^{1}\}$ which satisfy the following conditions:

\begin{enumerate}
\item[(1)] $s(e)e=e=er(e)$ for all $e\in E^{1}$.

\item[(2)] $r(e)e^{\ast}=e^{\ast}=e^{\ast}s(e)$\ for all $e\in E^{1}$.

\item[(3)] (CK-1 relations) For all $e,f\in E^{1}$, $e^{\ast}e=r(e)$ and
$e^{\ast}f=0$ if $e\neq f$.

\item[(4)] (CK-2 relations) For every regular vertex $v\in E^{0}$,
\[
v=\sum_{e\in E^{1},\ s(e)=v}ee^{\ast}.
\]

\end{enumerate}
\end{definition} 
The Leavitt path algebra is spanned as a $K$-vector space by the set of monomials 
$$\{\gamma \lambda^{\ast} |\gamma , \lambda \in Path(E) \text{ such that } r(\gamma) = r(\lambda) \}$$
That is, any $x \in L_K(E)$,
$$ x=\displaystyle\sum_{i=1}^{n}k_{i}\gamma_{i}\lambda_{i}^{\ast} \quad \text{for any} \quad k_{i}\in K , \gamma_{i},\lambda_{i} \in Path(E).$$

Some familiar rings appear as examples of Leavitt path algebras, for instance:   

\begin{example}Take the graph $E$ as
$$\xymatrix{{\bullet}^{v_1} \ar [r] ^{e_1} & {\bullet}^{v_2}  \ar@{.}[r] & {\bullet}^{v_{n-1}} 
	\ar [r]^{e_{n-1}} & {\bullet}^{v_n}}$$
$L_K(E)\cong M_n(K)$. 
\end{example}
\begin{example} \label{loop} Take the graph $R_1$ as
$$  \xymatrix{{\bullet}^{v} \ar@(ur,dr)^e }$$
In this case, $L_K(R_1)\cong K[x,x^{-1}]$ via $ v \mapsto 1 , e \mapsto x , e^{\ast} \mapsto x^{-1}.$
\end{example}

\begin{example} For $n \geq 2$, consider the graph 
$$ R_n  =  
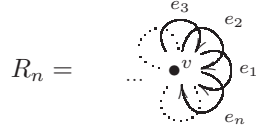
\xymatrix{ & {\bullet^v} \ar@(ur,dr) ^{e_1} \ar@(u,r) ^{e_2}
	\ar@(ul,ur) ^{e_3} \ar@{.} @(l,u) \ar@{.} @(dr,dl) \ar@(r,d) ^{e_n}
	\ar@{}[l] ^{\ldots} }$$
Then  $L_K(R_n)\cong L_K(1,n)$ which is Leavitt algebra of type $(1,n)$. 
\end{example}
Recall that a ring $R$ is said to have \textit{a set of local units} $F$, where $F$ is a set of idempotents
in $R$ having the property that, for each finite subset $r_1 ,\ldots,r_n$ of $R$, there exists $f \in F$ with $fr_if=r_i$ for
all $1 \leq i \leq n$. A ring $R$ with unit 1 is, clearly, a ring with a set of local units where $F=\{ 1 \}$. 

In the case of Leavitt path algebras, 
for each $x \in L_K(E)$ there exists a finite set of distinct vertices $V(x)$ for which $x = f x f$ , where
$f = \sum_{v \in V(x)} v$. When  $E^0$ is finite, $L_{K}(E)$ is a ring with unit element 
$\displaystyle 1= \sum_{v\in E^{0}}v$. Otherwise, $L_{K}(E)$ is not a unital ring, but is a ring with local units 
consisting of sums of distinct elements of $E^0$.

One of the most important properties of the class of Leavitt path algebras is that each $L_K(E)$ is a $\mathbb{Z}$-graded $K$-algebra. 
that is, 
$L_{K}(E)={\displaystyle\bigoplus\limits_{n\in\mathbb{Z}}}L_{n}$ 
induced by defining, for all $v\in E^{0}$ and $e\in E^{1}$, $\deg
(v)=0$, $\deg(e)=1$, $\deg(e^{\ast})=-1$. Further, for each 
$n\in \mathbb{Z}$, the homogeneous component $L_{n}$ is given by
\[L_{n}=\left\{{\textstyle\sum}
k_{i}\alpha_{i}\beta_{i}^{\ast}\in L:\text{ }l(\alpha_{i})-l(\beta
_{i})=n, \ k_i \in K, \ \alpha_{i},\beta_{i} \in Path(E)\right\}.
\]
An ideal $I$ of $L_{K}(E)$ is said to be a \textit{graded ideal} if 
$I={\displaystyle\bigoplus\limits_{n\in\mathbb{Z}}}(I\cap L_{n})$.
In the sequel all ideals of our concern will be two-sided.
\section{Ideals in Leavitt Path Algebras} 
 Recall that an (not necessarily unital) algebra $R$ is called {\it simple}, if $R^2 \neq 0$ and $R$ has no proper non-trivial ideals. Simple Leavitt path algebras are characterized in \cite{AA} by Abrams and Aranda Pino. 
\begin{theorem}
	Let $E$ be an arbitrary graph and $K$ be any field. Then $L_K(E)$ is simple if and only if $E$ has Condition $(L)$ and the only hereditary saturated subsets of $E^0$ are $\emptyset$ and $E^0$.
\end{theorem}

In a Leavitt path algebra, the intersection of any ideal with the set of vertices is always a hereditary set. 
\begin{lemma}(\cite[Lemma 3.9]{AA})
	Let $E$ be an arbitrary graph and $K$ be any field. Let $N$ be an ideal of $L_K(E)$. Then $N \cap E^0 \in \mathscr{H}_E$.
\end{lemma}
$N \cap E^0$ may very well be the empty set, however 
if the Leavitt path algebra is over a graph that satisfies Condition $(L)$ then $N$ definitely contains a vertex (an idempotent).  
\begin{proposition} (\cite[Corollary 3.8]{AA})
	Let $E$ be an arbitrary graph and  
	Let $E$ be a graph satisfying Condition $(L)$ and $K$ be any field. Then every nonzero ideal of $L_K(E)$ contains a vertex.
\end{proposition}
\begin{proposition} 
Let $E$ be an arbitrary graph and $K$ be any field. Let $H$ be a hereditary subset of $E^0$. Then there is a $\mathbb{Z}$-graded monomorphism $\varphi$ from $L_K(E_H)$ into $L_K(E)$ via
$v \mapsto v , e \mapsto e , e^{\ast} \mapsto e^{\ast} $ for all $ v \in E_{H}^{0} ,  e \in E_{H}^{1}$.
\end{proposition}

We give a description of the elements in the ideal generated by a hereditary subset of vertices.
\begin{lemma}(\cite[Lemma 5.6]{T}) 
	Let $E$ be an arbitrary graph and $K$ be any field. 
	
	(i) Let $H$ be a hereditary subset of $E^0$. Then the ideal $I(H)$ is
	\[
	I(H)= span_K (\{\gamma \lambda^{\ast} |\gamma , \lambda \in Path(E) \text{ such that } r(\gamma) = r(\lambda) \in H \})\]
	\[
	=\Big \{ \displaystyle\sum_{i=1}^{n}k_i \gamma_i \lambda_{i}^{\ast} | n \geq 1 , k_i \in K, \gamma_i , \lambda_i \in Path(E) \text{ such that } r(\gamma_i) = r(\lambda_i) \in H \Big\}
	\]
	(ii) Let $H$ be a hereditary subset of $E^0$ and $S$ a subset of $B_H$. Then the ideal 
	$$I(H, S)=span_K (\{\gamma \lambda^{\ast} |\gamma , \lambda \in Path(E) \text{ such that } r(\gamma) = r(\lambda) \in H \})$$
	 $$ + span_K (\{\alpha v^H \beta^{\ast} | \alpha , \beta \in Path(E) \text{ and } v \in S \}).$$
	 
	\end{lemma}

\subsection{Graded Ideals}
First, we mention the result on graded simplicity, that is when 
$L_K(E)$ has no non-trivial graded ideals. As stated in \cite[Cor.2.5.15]{AAS}, $L_K(E)$ is graded simple if and only if the only hereditary saturated subsets of $E^0$ are $\emptyset$ and $E^0$.
A typical example of a graded simple Leavitt path algebra is $K[x, x^{-1}]$, see Example \ref{loop}. 
However, since $\langle 1+x \rangle$ is a (non-graded) ideal, $K[x, x^{-1}]$ is not simple. Hence, it is possible to have non-trivial non-graded ideals in a graded simple ring.

Now, we are ready to describe the graded ideals in Leavitt path algebras which is in \cite[Remark 2.2]{APS1}.
\begin{theorem}
	Let $E$ be an arbitrary graph and $K$ be any field. Then every graded ideal $N$ of $L_K(E)$ is generated by $H \cup S^H$, where $H=N \cap E^0 \in \mathscr{H}_E$, and $S=\{ v \in B_H | v^H \in N \}$ , i.e. $N=I(H,S)$.
	
	In particular, every graded ideal of $L_K(E)$ is generated by a set of homogeneous idempotents.
\end{theorem}

Observe that if $N = I(H,S)$ is a graded ideal, so that $N= \langle H,v^H:v \in S \rangle$, the generators $u$ in $H$ and $v^H$ are all idempotents. So they all belong to $N^2$, that is if $N$ is a graded ideal, then $N = N^2$. Conversely, if $N$ is an ideal such that 
$N = N^2$, we use a result from \cite{EKR}. In \cite[Theorem 3.6]{EKR}, it was shown that for any ideal $N$, 
the intersection of $\{N^n: n > 0\}$ is a graded ideal. So, if $N^2 = N$, then $N =  \cap \{N^n \ :\  n > 0\}$ is a graded ideal. Thus we obtain the following characterization of graded ideals of a Leavitt path algebra (which also appears in \cite[Cor. 2.9.11]{AAS} via a different proof.)
\begin{theorem} 
Let $E$ be an arbitrary graph and $K$ be any field. Then, an ideal $N$ of $L_K(E)$ is graded if and only if $N^2 = N$.
\end{theorem}
\medskip

The correspondence between the quotient Leavitt path algebra and the Leavitt path algebra of the quotient graph is noteworthly to state at this point. Part (i) of the following theorem appears as \cite[Lemma 2.3]{APS1} and part (ii) appears in \cite[Theorem 5.7]{T}.
\begin{theorem}
	Let $K$ be any field,
	\begin{enumerate}
		\item [(i)]  $E$ be a row-finite graph, and $H \in \mathscr{H}_E$. Then $L_K(E) / I(H) \cong L_K(E/H)$ as $\mathbb{Z}$-graded $K$-algebras.
		\item [(ii)] $E$ be an arbitrary graph,  $H \in \mathscr{H}_E$ and $S \subset B_H$. Then $L_K(E)/I(H,S) \cong L_K(E/(H,S))$ as 
	$\mathbb{Z}$-graded $K$-algebras.	
		\end{enumerate}
\end{theorem}

\subsection{The Structure Theorem of Graded Ideals}
Now, we are ready to give a complete description of the lattice of graded ideals of a Leavitt path algebra in terms of specified subsets of $E^0$, that is the Structure Theorem for Graded Ideals. The results in this section first appeared for row-finite graphs in \cite{AMP} and 
for arbitrary graphs in \cite{T}.
\begin{definition}\rm
	Let $E$ be an arbitrary graph and $K$ be any field. Denote $\mathscr{L}_{gr}(L_K(E))$ the lattice of graded ideals of $L_K(E)$, whose order is inclusion, also supremum and infimum are the usual operations of ideal sum and intersection.
\end{definition}
\begin{remark}\rm
	Let $E$ be an arbitrary graph. We define in $\mathscr{H}_E$ a partial order by setting $ H \leq H' $ in case $ H \subseteq H' $. So, $\mathscr{H}_E$ is a complete lattice, with supremum $\vee$ and infimum $\wedge$ in $\mathscr{H}_E$ given by setting $\vee_{i \in \Gamma} H_i := \overline{\cup_{i \in \Gamma} H_i}$ and $\wedge_{i \in \Gamma} H_i := \cap_{i \in \Gamma} H_i$ respectively.
\end{remark}
\begin{definition}\rm
	Let $E$ be an arbitrary graph. We set 
	$$\mathscr{S}=\displaystyle\bigcup_{H \in \mathscr{H}_E} \mathscr{P}(B_H),$$
	where $\mathscr{P}(B_H)$ denotes the set of all subset of $B_H$. 
	
	We denote by $\mathscr{T}_E$ the subset of $\mathscr{H}_E \times \mathscr{S}$ consisting of pairs of the form $(H,S)$, where $S \in \mathscr{P}(B_H)$. We define in $\mathscr{T}_E$ the following relation:
	$$ (H_1,S_1) \leq (H_2,S_2) \text{ if and only if } H_1 \subseteq H_2 \text{ and } S_1 \subseteq H_2 \cup S_2 .$$

\end{definition}
\begin{proposition}
	Let $E$ be an arbitrary graph. For $(H_1,S_1),(H_2,S_2) \in \mathscr{T}_E$, we have 
	$$(H_1,S_1) \leq (H_2,S_2) \Longleftrightarrow I(H_1, S_{1}) \subseteq I(H_2,S_{2}).$$
	In particular, $\leq$ is a partial order on $\mathscr{T}_E$.
\end{proposition}

Fore more details on the lattice structure of $\mathscr{T}_E$, see \cite{AAS}.

\begin{theorem} (\cite[Theorem 5.7]{T})
	%\textbf{(The Structure Theorem for Graded Ideals)} 
	Let $E$ be an arbitrary graph and $K$ be any field. Then the map $\varphi$ given here provides a lattice isomorphism:
	$$ \varphi : \mathscr{L}_{gr}(L_K(E)) \to \mathscr{T}_E \quad \text{ via } \quad I \mapsto (I \cap E^0,S).$$
	where $S=\{v \in B_H | v^H \in I\}$ for $H=I \cap E^0$. The inverse $\varphi'$ of $\varphi$ is given by:
	$$\varphi' : \mathscr{T}_E \to \mathscr{L}_{gr}(L_K(E)) \quad \text{ via } \quad (H,S) \mapsto I(H \cup S^H).$$
\end{theorem} 
\begin{theorem}(\cite[Theorem 5.3]{AMP})
	Let $E$ be a row-finite graph and $K$ be any field. The following map $\varphi$ provides a lattice isomorphism:
	$$ \varphi : \mathscr{L}_{gr} (L_K(E)) \to \mathscr{H}_E \quad \text{ via } \quad \varphi(I)=I \cap E^0,$$
	with inverse given by
	$$ \varphi' : \mathscr{H}_E \to \mathscr{L}_{gr} (L_K(E)) \quad \text{ via } \quad \varphi'(H)=I(H).$$
\end{theorem}

Let $E$ be an arbitrary graph and $K$ be any field. Then every graded ideal of $L_K(E)$ is $K$-algebra isomorphic to a Leavitt path algebra Part (i) of the following theorem first appears in \cite[Lemma 5.2]{APS1} under the hypothesis that graph $E_H$ satisfies Condition (L).
\begin{theorem}
	Let $E$ be an arbitrary graph and $K$ be any field. Let $H$ be a non-empty hereditary subset of $E$ and $S \subseteq B_H$. Then 
	(i) $I(H)$ is $K$-algebra isomorphic to $L_K(_H E)$; \\
	(ii) $I(H, S)$ is isomorphic as $K$-algebras to $L_K(_{(H,S)}E)$.
\end{theorem}

\subsection{Structure of Two-Sided Ideals}
The generators of an ideal are studied in \cite{Ranga2} and gives a useful characterization of the graded and non-graded part of an ideal. 
The following results are due to Rangaswamy and finally achieving that in a Leavitt path algebra, any finitely generated ideal is principal \cite{Ranga2}. 
\begin{theorem}
	Let $E$ be an arbitrary graph and $K$ be any field. Then any non-zero ideal of the $L_K(E)$ is generated by elements of the form
	$$\bigg( u+ \displaystyle\sum_{i=1}^{k}k_ig^{ri}\bigg)\Big(u- \displaystyle\sum_{e\in X}ee^{\ast} \Big)$$
	where $ u \in E^0$, $k_i \in K$,$r_i$ are positive integers, $X$ is a finite (possibly empty) proper subset of $s^{-1}(u)$ and, 
	whenever $k_i \neq 0$ for some $i$, then $g$ is a unique cycle based at $u$. 
\end{theorem}
The main result of \cite{Ranga2} is the following theorem: 
\begin{theorem}
		Let $I$ be an arbitrary nonzero ideal of $L_K(E)$ with $I \cap E^0=H$ and $S=\{v \in B_H : v^H \in I \}$. Then $I$ is generated by $H \cup \{v^H : v \in S \} \cup Y$ where $Y$ is a set of mutually orthogonal elements of the form $(u+\sum_{i=1}^{n}k_ig^{r_i})$ in which the following statements hold:
		\begin{enumerate}
			\item [(i)] $g$ is a (unique) cycle with no exits in $E^0 \backslash H$ based at a vertex $u$ in $E^0 \backslash H$; and 
			\item [(ii)] $k_i \in K$ with at least one $k_i \neq 0$.
		\end{enumerate}
	If $I$ is nongraded, then $Y$ is nonempty.
\end{theorem}

\begin{corollary}
		Every finitely generated ideal of $L_K(E)$ is a principal ideal. Moreover, if $E$ is a finite graph, then every ideal is principal.
\end{corollary}

\subsection{Prime and Primitive Ideals}
The structure of prime ideals has played a key role in ring theory. In the Leavitt path algebra setting the first paper to focus on the prime and primitive ideals of  Leavitt path algebras on row-finite graphs has been  \cite{APS}.  Later the prime ideal structure on an arbitrary graph was studied in \cite{Ranga1}, while the primitive Leavitt path algebras are described in \cite{ABR}. The primitive algebras have also been important as a consequence of Kaplansky's question:  "Is a regular prime ring necessarily primitive?" 

We recall a few ring-theoretic definitions. A two-sided ideal $P$ of a ring $R$ is \textit{prime} in case $P\neq R$ and $P$ has the property that for any two-sided ideals $I,J$ of  $R$, if $IJ \subseteq P$ then either $I \subseteq P$ of $J \subseteq P$. The ring $R$ is called \textit{prime} in case $\{0\}$ is a prime ideal of $R$. It is easily shown that $P$ is a prime idal of $R$ if and only if $R/P$ is a prime ring. The set of all prime ideals of $R$ is denoted by Spec$(R)$, call the prime spectrum of $R$. A ring $R$ is called \textit{left primitive} if $R$ admits a simple faithful left $R$-module. It is easy to show that any primitive ring is prime.

A ring is  \textit{von Neumann regular} (or regular) in case for each $a \in R$ there exists $x \in R$ for which $a = axa$. In the theory of Leavitt path algebras the necessary and sufficient condition for $L_K(E)$ to be regular is given by Abrams and Rangaswamy \cite{AR}.
\begin{theorem}
 Let $E$ be an arbitrary graph and $K$ be any field. $L_K(E)$ is von Neumann regular if and only if $E$ is acyclic.
\end{theorem}
  
Recall the graph the one vertex, one loop graph $R_1$ of the Example \ref{loop}. The prime ideals of the principal ideal domain $K[x,x^{-1}] \cong L_K(R_1)$ provide a model for the prime spectra of general Leavitt path algebras. The key property of $R_1$ in this setting is that it contains a unique cycle without exits. Specifically, Spec$(K[x,x^{-1}])$ consists of the ideal $\{0\}$, together with ideals generated by the irreducible polynomials of $K[x,x^{-1}]$. The irreducible polynomials are of the form $x^n f(x)$, where $f(x)$ is an irreducible polynomial in the standard polynomial ring $K[x]$, and $n \in \mathbb{Z}$. In particular, there is exactly one graded prime ideal (namely,$\{0\}$) in $L_K(R_1)$. All the remaining prime ideals of $L_K(R_1)$ are non-graded corresponding to irreducible polynomials in $K[x,x^{-1}]$.

The prime ideals of a Leavitt path algebra are completely characterized in the following theorem. Recall that $M(u)$ is defined in Definition 
\ref{M(u)}.
\begin{theorem} \cite[Thm.3.12]{Ranga1} \label{prime}
Let $E$ be an arbitrary graph and $K$ be any field. Let $P$ be an ideal of $L_K(E)$ with $P \cap E^0 =H$. Then $P$ is a prime ideal of $L_K(E)$ if and only if $P$ satisfies one of the following conditions:
\begin{enumerate}
	\item [(i)] $P=\langle H, \{v^H:v \in B_H\}\rangle$ and $E^0 \backslash H$ satisfies the $MT-3$ condition;
	\item [(ii)] $P=\langle H, \{v^H:v \in B_H \backslash \{u\}\}\rangle$ for some $u \in B_H$ and $E^0 \backslash H=M(u)$;
	\item [(iii)] $P=\langle H, \{v^H:v \in B_H\},f(c)\rangle$ where $c$ is a cycle without $K$ in $E$ based at a vertex $u$, $E^0 \backslash H=M(u)$ and $f(x)$ is an irreducible polynomial in $K[x,x^{-1}]$.
\end{enumerate}
\end{theorem}
Recall that a ring $R$ is prime if $\{0\}$ is a prime ideal, hence the immediate corollary to Theorem \ref{prime} follows.
\begin{corollary}
		Let $E$ be an arbitrary graph and $K$ any field. Then $L_K(E)$ is prime if and only if $E$ is $MT-3$.
\end{corollary}
When $E$ is row-finite, the characterization of a primitive $L_K(E)$ is given \cite{APS}.
\begin{theorem}
	%LPA BOOK  %%%%%%%%%%%%%%
	Let $E$ be a row-finite graph and $K$ be any field. Then $L_K(E)$ is primitive if and only if $E$ is $MT-3$ and Condition(L).
	\end{theorem}
When $E$ is an arbitrary graph, the result requires a new condition on the graph \cite{ABR}.
\begin{theorem}
	Let $E$ be any graph and $K$ be any field. Then $L_K(E)$ is primitive if and only if $E$ has $MT-3$, Condition(L) and Countable Separation Property.
\end{theorem}
We pause here to construct a Leavitt path algebra which is a counter example to Kaplansky's question "Is a regular prime ring necessarily primitive?", (see \cite{ABR} for details).
\begin{example}
	$X$ uncountable, $S$ the set of finite subsets of $X$. Define the graph $E$:
\begin{enumerate}
	\item 	[(1)] vertices indexed by $S$, and
	\item 	[(2)] edges induced by proper subset relationship.
\end{enumerate}
	Then $L_K(E)$ is regular, prime, not primitive.
\end{example}
The following results are from \cite{Ranga1}.
\begin{lemma}(\cite[Lemma 3.8]{Ranga1})
	Let $P$ be a prime ideal of $L_K(E)$ with $H=P \cap E^0$ and let $S=\{v \in B_H:v^H \in P\}$. Then the ideal $I(H,S)$ is also a prime ideal of $L_K(E)$.
\end{lemma}
\begin{corollary}(\cite[Corollary 3.9]{Ranga1})
	Let $E$ be an arbitrary graph and $K$ be any field. Then the Leavitt path algebra $L_K(E)$ is a prime ring if and only if there is a prime ideal of $L_K(E)$ which does not contain any vertices.
\end{corollary}

A natural question that arose is to answer the graded version of Kaplansky's question, namely 
whether every graded prime von Neumann regular Leavitt path algebra is graded primitive. 
This question is solved by the recent unpublished work of Rangaswamy \cite{Ranga3}. 
\begin{theorem}
For any arbitrary graph $E$ given, the following are equivalent 
\begin{enumerate}
	\item [(i)] $L_K(E)$ is graded primitive;
	\item [(ii)] $E^0$ is countably directed;
	\item [(iii)] $L_K(E)$ is graded prime and, for some vertex $v \in E^0$, the tree $T(v)$ satisfies the Countable Separation Property.
\end{enumerate}
\end{theorem}
The author in \cite{Ranga3}, provides many examples of graded von Neumann regular rings which are graded prime but not graded primitive. 

%and, if this is not the case, whether necessary and sufficient conditions can be given under which a graded prime Leavitt path algebra becomes a graded primitive ring.

\subsection{Maximal Ideals}
This section is quoted from \cite{EK} by Esin and the first named author.  

In a unital ring, any maximal ideal is also a prime ideal. However, this is not necessarily true for a non-unital ring.  
Consider, for instance, the non-unital ring $2\mathbb{Z}$, and its ideal $4\mathbb{Z}$. Notice that $4\mathbb{Z}$ is a maximal ideal, but not prime ideal in $2\mathbb{Z}$. The Leavitt path algebra is a unital ring, only if $E^0$ is finite. So it is worthwhile to study the maximal ideals in a non-unital setting. The following argument on maximal and prime ideals in non-unital Leavitt path algebras appears in \cite[pp.86-87]{Ranga1}.
\begin{proposition}	In a ring $R$ satisfying $R^2=R$, any maximal ideal is a prime ideal. Hence, in any Leavitt 
	path algebra, any maximal ideal is a prime ideal.
\end{proposition}
\begin{proof}
Suppose $R^2=R$, and let $M$ be a maximal ideal of $R$ such that $A \nsubseteq M$ and $B \nsubseteq M$ for some ideals $A,B$ of $R$. 
Then $R=R^2=(M+A)(M+B)= M^2 + AM + MB + AB \subseteq M + AB$. Then $M +AB =R$, and $AB \nsubseteq M$. 
Thus $M$ is a prime ideal.  
Now, since any Leavitt path algebra $R$ is a ring with local units, $R^2=R$ is satisfied and the result holds.
\end{proof}

As stated in \cite[Lemma 3.6]{Ranga1}, in a Leavitt path algebra $L_K(E)$, the largest graded ideal contained in any ideal $N$ (which is denoted by $gr(N)$) is the ideal generated by the admissible pair $(H,S)$ where $H=N \cap E^0$, and $S=\{ v \in B_H | v^H \in N\}$, i.e. 
$gr(N) = I(H,S)$.
One useful observation is that: if a non-graded ideal $N$ is a maximal element in $\mathscr{L}(L_K(E))$, the lattice of all two-sided ideals of a Leavitt path algebra, then $gr(N)$ is a maximal element in $\mathscr{L}_{gr}(L_K(E))$, the lattice of all two-sided graded ideals of this Leavitt path algebra (e.g. Example \ref{Maxgraded-nongraded-NoK}).

Maximal ideals always exist in a unital ring; however, this is not always true in a non-unital ring. 
Consider the Leavitt path algebra of the next example: 

\begin{example} \label{ikikulak}
\rm
Let $E$ be the row-finite graph with $E^{0}=\{v_{i} :i=1,2,\ldots \}$ and for each $i$, there is an edge $e_{i}$ with
$r(e_{i})=v_{i}$, $s(e_{i})=v_{i+1}$, also at each $v_{i}$ there are two
loops $f_{i},g_{i}$ so that $v_{i}=s(f_{i})=r(f_{i})=s(g_{i})=r(g_{i})$: 
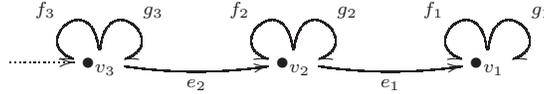
\[ 
\xymatrix{ \ar@{.>}[r] &
\bullet_{v_3}\ar@(u,l)_{f_3} \ar@(u,r)^{g_3} \ar@/_.3pc/[rr]_{e_2} & &  
\bullet_{v_2}\ar@(u,l)_{f_2} \ar@(u,r)^{g_2} \ar@/_.3pc/[rr]_{e_1}  && \bullet_{v_1}\ar@(u,l)_{f_1} \ar@(u,r)^{g_1} }
\]
The non-empty proper hereditary saturated subsets of vertices in $E$ are the sets $H_{n}=\{v_{1},\ldots,v_{n}\}$ 
for some $n\geq 1$ and they form an infinite chain under set inclusion. 
Graph $E$ satisfies Condition (K), so all ideals are graded, generated by $H_n$ for some $n$ and 
they form a chain under set inclusion. As the chain of ideals does not terminate, $L_{K}(E)$ does not contain any maximal ideals.
Note also that, $E^{0}\backslash(H_{n},\emptyset)$ is $MT-3$ for each $n$, 
thus all ideals are prime ideals.
\end{example}

A well-established question is to find out when a maximal ideal exist in a non-unital Leavitt path algebra. The necessary and sufficient condition depends on the existence of a maximal hereditary and saturated subset of $E^0$ as proved in \cite{EK}.  
\begin{theorem} 
(Existence Theorem) $L_K(E)$ has a maximal ideal if and only if $ \mathscr{H}_E$ has a maximal element.
\end{theorem}
\begin{proof} (Sketch: see \cite{EK} for details)
Assume $L_K(E)$ has a maximal ideal $M$, then there are two cases: \\
if $M$ is a graded ideal, then $M= I(H,S)$ for some $H \in \mathscr{H}_E$ and $S=\{ v \in B_H | v^H \in M\}$. However, 
$M= I(H,S) \leq I(H,B_H)$, and as $M$ is a maximal ideal, $S = B_H$. Then it can be shown that:  
$I(H,B_H)$ is a maximal ideal in $L_K(E)$ if and only if $H$ is a maximal element in $\mathscr{H}_E$ and the quotient graph $E \backslash (H,B_H)$ has Condition$(L)$.

If $M$ is a non-graded maximal ideal, then $gr(M) = I(H,S)$ is a maximal graded ideal where $H=M \cap E^0$, and 
$S=\{ v \in B_H | v^H \in M\}$. Similarly since $gr(M)$ is maximal, $S = B_H$. Again, it can be shown that:     
$H$ is a maximal element in $\mathscr{H}_E$ with $E \backslash (H,B_H)$ not satisfying Condition$(L)$, if and only if there is a maximal non-graded ideal $M$ containing $I(H,B_H)$ with $H=M \cap E^0$.

This completes the proof.
\end{proof}

Moreover, the poset structure of $ \mathscr{H}_E$ determines whether every ideal of the Leavitt path algebra 
is contained in a maximal ideal. 
\begin{theorem} The following assertions are equivalent: 
\begin{itemize}
\item [(i)] Every element $X \in \mathscr{H}_E$ is contained in a maximal element $Z \in \mathscr{H}_E$. 
\item [(ii)] Every ideal of $L_K(E)$ is contained in a maximal ideal.
\end{itemize}
\end{theorem}

\begin{example}\label{MaxgradedNoK} 
	\rm
	Let $E$ be the graph 
	\medskip
	
	$$ 
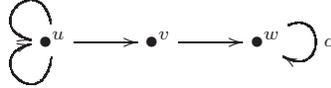
\xymatrix{
		{\bullet}^u  \ar@(u,l) \ar@(d,l)   \ar@{->}[r]  & {\bullet}^v \ar@{->}[r]
		& {\bullet}^w \ar@(ur,dr)^{c} }
	$$ 
	\medskip
	
	Then $E$ does not satisfy Condition (K), so the Leavitt path algebra on $E$ has 
	both graded and non-graded ideals. Let $Q$ be the graded ideal generated by the hereditary saturated set $H=\{v,w\}$. $Q$ is a maximal ideal as $L/Q$ is isomorphic to $L_K(E \backslash H)$ which is also isomorphic to the simple Leavitt algebra $L(1,2)$ (See the Example \ref{Ekornek}). By using Theorem \ref{prime}, we classify the prime ideals in $L$. There are infinitely many non-graded prime ideals each generated by $f(c)$ where $f(x)$ is an irreducible polynomial in $K[x,x^{-1}]$ which are all contained in $Q$. Also, the trivial ideal $\{ 0\}$ is prime as $E$ satisfies condition $MT-3$ and $L_K(E)$ has a unique maximal element $Q$. 
\end{example}

\medskip

We now give an example of a graph with infinitely many hereditary saturated sets and the corresponding Leavitt path algebra has 
a unique maximal ideal which is graded.

\begin{example} \label{tersikikulak}
	\rm
	Let $E$ be a graph with $E^{0}=\{v_{i} :i=1,2,\ldots \}$. For each $i$, there is an edge $e_{i}$ with
	$s(e_{i})=v_{i}$ and $r(e_{i})=v_{i+1}$ and at each $v_{i}$ there are two
	loops $f_{i},g_{i}$ so that $v_{i}=s(f_{i})=r(f_{i})=s(g_{i})=r(g_{i})$. Thus
	$E$ is the graph
	
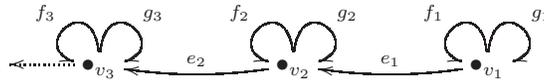
\[ 
	\xymatrix{  & \ar@{.>}[l]
		\bullet_{v_3}\ar@(u,l)_{f_3} \ar@(u,r)^{g_3} %\ar@/_.3pc/[rr]_{e_2} 
		& & \bullet_{v_2}\ar@(u,l)_{f_2} \ar@(u,r)^{g_2} \ar@/^.3pc/[ll]_{e_2}  && \bullet_{v_1}\ar@(u,l)_{f_1} \ar@(u,r)^{g_1} \ar@/^.3pc/[ll]_{e_1}}
	\]
	Now $E$ is a row-finite graph and the non-empty proper hereditary
	saturated subsets of vertices in $E$ are the sets $H_{n}=\{v_{n},v_{n+1},\ldots\}$ 
	for some $n\geq 2$ and $H_{n+1} \subsetneq H_n$ form an infinite chain under set inclusion and  
	$H_{2}=\{v_{2},v_{3},\ldots\}$ is the maximal element in $\mathcal{H}_E$. 
	The graph $E$ satisfies Condition (K), so all ideals are graded, generated by $H_n$ for some $n$. 
	So $L_{K}(E)$ contains a unique maximal ideal $I(H_2)$. 
	Note also that, $E^{0}\backslash H_{n}$ is $MT-3$ for each $n$, 
	thus all ideals of $L$ are prime ideals.
\end{example}

In a Leavitt path algebra, if a unique maximal ideal exists, then it is a graded ideal. Also, every maximal ideal is graded in $L_K(E)$ if and only if for every maximal element $H$ in $\mathscr{H}_E$, $E \backslash (H,B_H)$ satisfies Condition$(L)$. Note that there are Leavitt path algebras with both graded and non-graded maximal ideals as the following example illustrates. 

\begin{example}\label{Maxgraded-nongraded-NoK} 
\rm
Let $E$ be the graph 
\medskip
$$ 
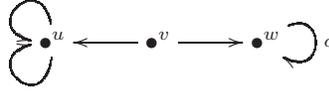
\xymatrix{
{\bullet}^u  \ar@(u,l) \ar@(d,l)    & {\bullet}^v \ar@{->}[r] \ar@{->}[l]
 & {\bullet}^w \ar@(ur,dr)^{c} }
$$ 
\medskip

Then the Leavitt path algebra on $E$ has both graded and non-graded maximal ideals. The set $\mathcal{H}_E$ is finite and hence any ideal is contained in a maximal ideal. The trivial ideal $\{ 0\}$ which is a graded ideal generated by the empty set, is not prime as $E$ does not satisfy condition $MT-3$. There are infinitely many non-graded prime ideals each generated by $f(c)$ where $f(x)$ is an irreducible polynomial in $K[x,x^{-1}]$ which all contain $\{ 0\}$.
Let $N$ be the graded ideal generated by the hereditary saturated set $H=\{u\}$ and in this case, the quotient graph $E \backslash H$ does not satisfy condition (L). Then there are infinitely many maximal non-graded ideals each generated by $f(c)$ where $f(x)$ is an irreducible polynomial in $K[x,x^{-1}]$ which all contain $N$. 
Also, let $Q$ be the graded ideal generated by the hereditary saturated set $H=\{w\}$. In this case, the quotient graph $E \backslash H$ satisfy condition (L). Hence, $Q$ is a maximal ideal. 
 
$L_K(E)$ has a infinitely many maximal ideals, one of them is graded, namely $Q$ and infinitely many are non-graded ideals whose graded part is $N$. 
\end{example}

It is an interesting question to answer when all non-zero prime ideals are maximal, as these rings are called rings with Krull dimension zero. 
In fact, Leavitt path algebras with prescribed Krull dimension are studied in \cite{Ranga1}. We conclude this article with two results from \cite{Ranga1}. 
\begin{theorem} \cite[Theorem 6.1]{Ranga1}
	Let $E$ be an arbitrary graph and $K$ be any field. Then every non-zero prime ideal of the Leavitt path algebra $L_K(E)$ is maximal if and only if $E$ satisfies one of the following two conditions:\\
	
	\textbf{Condition I:} (i) $E^0$ is a maximal tail; (ii) The only hereditary saturated subsets of $E^0$ are $E^0$ and $\emptyset$; (iii) $E$ does not satisfy the Condition$(K)$.\\
	
	\textbf{Condition II:} (a) $E$ satisfies the Condition$(K)$; (b) For each maximal tail M, the restricted graph $E_M$ contains no proper non-empty hereditary saturated subsets; (c) If $H$ is a hereditary saturated subset of $E^0$, then for each $u \in B_H, M(u) \subsetneqq E^0 \backslash H$
	\end{theorem}
	When $E$ is finite, the answer is much simpler. 
\begin{corollary}
	Let $E$ be a finite graph. Then every non-zero prime ideal of $L_K(E)$ is maximal if and only if either $L_K(E) \cong M_n(K[x,x^{-1}])$ for some positive integer $n$ or $E$ satisfies the Condition$(K)$ and, for each maximal tail $M$, the restricted graph $E_M$ contains no proper non-empty hereditary saturated subsets of vertices.
\end{corollary}

\end{document}